\documentclass[a4paper, 11pt]{article}

\usepackage{amssymb,amsmath}\usepackage{stmaryrd}
\usepackage{amsthm}

\usepackage[latin1]{inputenc}
\usepackage{subfigure}
\usepackage{bbm}
\usepackage{graphicx}

\newtheorem{theorem}{Theorem}
\newtheorem{lemma}[theorem]{Lemma}
\newtheorem{corollary}[theorem]{Corollary}
\newtheorem{proposition}[theorem]{Proposition}

\newtheorem{definition}[theorem]{Definition}

\newcommand{\Ind}{\mathbbm{1}}

\newcommand{\hm}{\widehat{\mathfrak{m}}}
\newcommand{\al}{\alpha}
\newcommand{\la}{\lambda}
\newcommand{\sig}{\sigma}
\newcommand{\si}{\sigma}
\newcommand{\be}{\beta}

\newcommand{\modu}[1]{\overline{#1}}
\newcommand{\half}{{\scriptscriptstyle{1\!/\!2}} }

\newcommand{\inter}{\bullet}
\newcommand{\eps}{\epsilon}
\newcommand{\mS}{\mathcal{S}}
\newcommand{\mN}{\mathcal{N}}
\newcommand{\NN}{\mathbb{N}}

\newcommand{\Cat}{\mathrm{Cat}}
\newcommand{\m}{\mathfrak{m}}
\newcommand{\n}{\mathfrak{n}}
\newcommand{\A}{{\bf A\/}}
\newcommand{\B}{{\bf B\/}}
\newcommand{\C}{{\bf C\/}}
\newcommand{\D}{{\bf D\/}}
\newcommand{\F}{{\bf F\/}}
\newcommand{\asc}{\textrm{asc}(\m)}
\newcommand{\dsc}{\textrm{dsc}(\m)}
\newcommand{\floor}[1]{\lfloor #1 \rfloor}

\newcommand{\TRL}{T_{\scriptscriptstyle{\mathrm{RL}}}}
\newcommand{\TLR}{T_{\scriptscriptstyle{\mathrm{LR}}}}
\newcommand{\s}{\mathfrak{s}}

\newcommand{\Eref}[1]{(\ref{#1})}
\usepackage{comment}

\usepackage[top=4cm, bottom=3cm, left=3cm, right=3cm]{geometry}

\begin{document}

\title{\bf Counting unicellular maps on non-orientable surfaces}

\author{
{\bf Olivier Bernardi\thanks{Supported by ANR project A3.}}\\
{\small Department of Mathematics,} \\
{\small Massachusetts Institute of Technology,}\\
{\small 77 Massachusetts Avenue,}\\
{\small Cambridge MA 02139, USA.}
\and 
{\bf Guillaume Chapuy\thanks{Supported by a CNRS/PIMS postdoctoral fellowship.
}} \\ 
{\small Department of Mathematics,}\\
{\small Simon Fraser University,} \\
{\small 8888 University Drive,} \\
{\small Burnaby B.C. V5A 1S6, Canada}.
}

\maketitle

\pagestyle{myheadings}
\markright{\small O. Bernardi, G. Chapuy -- Counting unicellular maps on non-orientable surfaces.}

\begin{abstract}
A unicellular map is the embedding of a connected graph in a surface in such a way that the complement of the graph is a topological disk. 
In this paper we present a bijective link between unicellular maps on a non-orientable surface and unicellular maps of a lower topological type, with distinguished vertices. From that we obtain a recurrence equation that leads to (new) explicit counting formulas for non-orientable unicellular maps of fixed topology. In particular, we give exact formulas for the precubic case (all vertices of degree $1$ or $3$), and asymptotic formulas for the general case, when the number of edges goes to infinity.
Our strategy is inspired by recent results obtained by the second author for the orientable case, but significant novelties are introduced: in particular we construct an involution which, in some sense,  ``averages'' the effects of non-orientability.\\
\end{abstract}

\section{Introduction}
A \emph{map} is an embedding of a connected graph in a (2-dimensional, compact, connected) surface considered up to homeomorphism. By \emph{embedding}, we mean that the graph is drawn on the surface in such a way that the edges do not intersect and the \emph{faces} (connected components of the complementary of the graph) are simply connected. Maps are sometimes referred to as \emph{ribbon graphs, fat-graphs}, and can be defined combinatorially rather than topologically as is recalled in Section~\ref{section:topo}. A map is \emph{unicellular} if is has a single face. For instance, the unicellular maps on the sphere are the plane trees.\\

In this paper we consider the problem of counting unicellular maps by the number of edges, when the topology of the surface is fixed.
In the orientable case, this question has a respectable history.
The first formula for the number $\eps_g(n)$ of orientable unicellular maps with $n$ edges and genus $g$ (hence $n+1-2g$ vertices) was given by Lehman and Walsh in \cite{Walsh:counting-maps-1}, as a sum over the integer partitions of size $g$.  Independently, Harer and Zagier found a simple recurrence formula for the numbers $\eps_g(n)$ \cite{Harer-Zagier}. Part of their proof relied on expressing the generating function of unicellular maps as a matrix integral. Other proofs of Harer-Zagier's formula were given in \cite{Lass:Harer-Zagier,Goulden:Harer-Zagier}. Recently, Chapuy~\cite{Chapuy:unicellular}, extending previous results for cubic maps~\cite{Chapuy:PTRF}, gave a bijective construction that relates unicellular maps of a given genus to unicellular maps of a smaller genus, hence leading to a new recurrence equation for the numbers $\eps_g(n)$. In particular, the construction in\cite{Chapuy:unicellular} gives a combinatorial interpretation of the fact that for each $g$ the number $\eps_g(n)$ is the product of a polynomial in $n$ times the $n$-th Catalan number 
$\Cat(n)=\frac{1}{n+1}{2n\choose n}$.\\

For non-orientable surfaces, results are more recent. The interpretation of matrix integrals over the Gaussian Orthogonal Ensemble (space of real symmetric matrices) in terms of maps was made explicit in  \cite{Goulden:maps-nonorientable+unicellular}. Ledoux~\cite{Ledoux:recursion-unicellular}, by means of matrix integrals and orthogonal polynomials, obtained for unicellular maps on general surfaces a recurrence relation which is similar to the Harer-Zagier one. As far as we know, no direct combinatorial nor bijective technique have successfully been used for the enumeration of a family of non-orientable maps until now.\\


A unicellular map is \emph{precubic} if it has only vertices of degree~$1$ and~$3$: precubic unicellular maps are a natural generalization of binary trees to general surfaces. 
In this paper, we give for all $h\in\frac{1}{2} \NN$ an explicit formula for the  number $\eta_h(m)$ of precubic unicellular maps of \emph{size} $m$ ($2m+\Ind_{h\in \NN}$ edges)  on the non-orientable surface of Euler Characteristic $2-2h$. These formulas (Corollaries~\ref{cor:integer} and~\ref{cor:noninteger}) take the form $\displaystyle \eta_h(m)=P_h(m)\Cat(m)$ if $h$ is an integer, and  $\displaystyle \eta_h(m)=P_h(m)4^m$ otherwise,
where $P_h$ is a polynomial of degree $3\lfloor h \rfloor$.
Our approach, which is completely combinatorial, is based on two ingredients. 
The first one, inspired from the orientable case~\cite{Chapuy:PTRF,Chapuy:unicellular}, is to consider some special vertices called \emph{intertwined nodes}, whose deletion reduces the topological type $h$ of a map. 
The second ingredient is of a different nature: we show that, among non-orientable maps of a given topology and size, the \emph{average number} of intertwined nodes per map can be determined explicitly. This is done thanks to an \emph{averaging involution}, which is described in Section~\ref{sec:averaging}. This enables us to find a simple recurrence equation for the numbers  $\eta_h(m)$. As in the orientable case, an important feature of our recurrence is that it is recursive \emph{only} on the topological type $h$, contrarily to equations of the Harer-Zagier type~\cite{Harer-Zagier,Ledoux:recursion-unicellular}, where also the number of edges vary. It is then easy to iterate the recurrence in order to obtain an explicit formula for $\eta_h(m)$.\\

In the case of \emph{general} (not necessarily precubic) unicellular maps, our approach does not work \emph{exactly}, but it does work \emph{asymptotically}.
That is, we obtain, with the same technique, the asymptotic number of non-orientable unicellular maps of fixed topology, when the number of edges tends to infinity (Theorem~\ref{thm:asymptotic}).
As far as we know, all the formulas obtained in this paper are new.\\

\section{Topological considerations}\label{section:topo}
In this section we recall some definitions on maps and gather the topological tools needed for proving our results. One of these tools is a canonical way to represent non-orientable maps combinatorially which will prove very useful for our purposes.\\


We denote $\NN=\{0,1,2,3,\ldots\}$ and $\frac{1}{2}\NN=\{0,\frac{1}{2},1,\frac{3}{2},\ldots\}$. For a non-negative real number $x$, we denote by $\floor{x}$ the integer part of $x$. For a non-negative integer $n$, we denote $n!!=n\cdot (n-2)!!$ if $n>1$, and $0!!=1!!=1$.

\subsection{Classical definitions of surfaces and maps}
{\bf \indent  Surfaces.} Our \emph{surfaces} are compact, connected, 2-dimensional manifolds. We consider surfaces up to homeomorphism. For any non-negative integer $h$, we denote by $\mS_h$ the torus of genus $h$, that is, the orientable surface obtained by adding $h$ \emph{handles} to the sphere. For any $h$ in $\frac{1}{2}\NN$, we denote by   $\mN_h$  the non-orientable surface obtained by adding $2h$ \emph{cross-caps} to the sphere. Hence, $\mS_0$ is the sphere, $\mS_1$ is the torus, $\mN_{1/2}$ is the projective plane and $\mN_1$ is the Klein bottle.  The \emph{type}  of  the surface $\mS_h$ or $\mN_h$ is the number $h$.
By the theorem of classification, each orientable surface is homeomorphic to one of the $\mS_h$ and each non-orientable surface is  homeomorphic to one of the $\mN_h$ (see e.g. \cite{Mohar:graphs-on-surfaces}).\\

\noindent {\bf Maps as graphs embedding.} Our \emph{graphs} are finite and undirected; loops and multiple edges are allowed. A \emph{map} is an embedding (without edge-crossings) of a connected graph into a surface, in such a way that the \emph{faces} (connected components of the complement of the graph) are simply connected. Maps are always considered up to homeomorphism. A  map is \emph{unicellular} if it has a single face.  

Each edge in a map is made of two \emph{half-edges}, obtained by removing its middle-point. The \emph{degree} of a vertex is the number of incident half-edges. A \emph{leaf} is a vertex of degree 1. 
A \emph{corner} in a map is an angular sector determined by a vertex, and two half-edges which are consecutive around it. The total number of corners in a map equals the number of half-edges which is twice the number of edges. A map is \emph{rooted} if it carries a distinguished half-edge called the \emph{root}, together with a distinguished side of this half-edge. 
The vertex incident to the root is the \emph{root vertex}.  The unique corner incident to the root half-edge and its distinguished side is the \emph{root corner}. 
\emph{From now on, all maps are rooted}.

The \emph{type} $h(\m)$ of a map $\m$ is the type of the underlying surface, that is to say, the Euler characteristic of the surface is $2-2h(\m)$. If $\m$ is a map, we let $v(\m)$, $e(\m)$ and $f(\m)$ be its numbers of vertices, edges and faces. These quantities satisfy the \emph{Euler formula}:
\begin{eqnarray}
e(\m)= v(\m)+f(\m)-2+2h(\m).
\end{eqnarray}

\noindent {\bf Maps as graphs with rotation systems and twists.} Let $G$ be a graph. To each edge $e$ of $G$ correspond two \emph{half-edges}, each of them incident to an endpoint of $e$ (they are both incident to the same vertex if $e$ is a loop). A \emph{rotation system} for $G$ is the choice, for each vertex $v$ of $G$, of a cyclic ordering of the half-edges incident to $v$.
We now explain the relation between maps and rotation systems. Our surfaces are locally orientable and an \emph{orientation convention} for a map $\m$ is the choice of an orientation, called \emph{counterclockwise orientation}, in the vicinity of each vertex. Any orientation convention for the map $\m$ induces a rotation system on the underlying graph, by taking the counterclockwise ordering of appearance of the half-edges around each vertex. 
Given an orientation convention, an edge $e=(v_1,v_2)$ of $\m$ is a \emph{twist} if the orientation conventions in the vicinity of the endpoints $v_1$ and $v_2$ are not simultaneously extendable to an orientation of a vicinity of the edge $e$; this happens exactly when the two sides of $e$ appear in the same order when crossed counterclockwise around $v_1$ and counterclockwise around $v_2$.
Therefore a map together with an orientation convention defines both a rotation system and a subset of edges (the twists). The \emph{flip} of a vertex $v$ consists in inverting the orientation convention at that vertex. This changes the rotation system at $v$ by inverting the cyclic order on the half-edges incident to~$v$, and changes the set of twists by the fact that non-loop edges incident to $e$ become twist if and only if they were not twist (while the status of the other edges remain unchanged). 
The next lemma is a classical topological result (see e.g. \cite{Mohar:graphs-on-surfaces}).
\begin{lemma}
A map (and the underlying surface) is entirely determined by the triple consisting of its (connected) graph, its rotation system, and the subset of its edges which are twists. Conversely, two triples define the same map if and only if one can be obtained from the other by flipping some vertices. 
\end{lemma}
By the lemma above, we can represent maps of positive types on a sheet of paper as follows: we draw the graph (with possible edge crossings) in such a way that the rotation system at each vertex is given by the counterclockwise order of the half-edges, and we indicate the twists by marking them by a cross (see e.g. Figure~\ref{fig:graph}).  The faces of the map are in bijection with the \emph{borders} of that drawing, which are obtained by walking along the edge-sides of the graph, and using the crosses in the middle of twisted edges as ``crosswalks'' that change the side of the edge along which one is walking (Figure~\ref{fig:graph}). Observe that the number of faces of the map gives the type of the underlying surface using Euler's formula.
\begin{figure}[h]
 \begin{minipage}{.58\linewidth}
\centerline{\includegraphics[scale=.8]{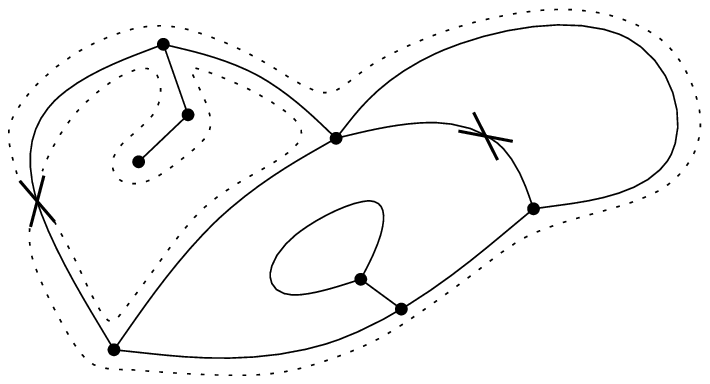}}
\caption{A representation of a map on the Klein bottle with three faces. The border of one of them is distinguished in dotted lines.}
\label{fig:graph}
 \end{minipage}\hfill 
\begin{minipage}{.38\linewidth}
\vspace{2.7mm}
\centerline{\includegraphics[scale=1]{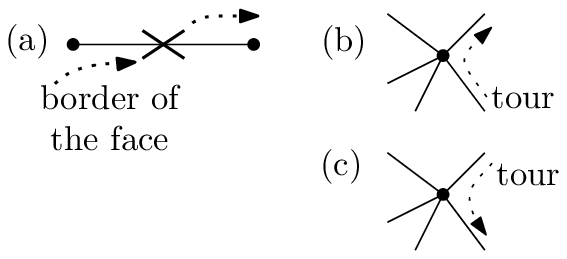}}
\vspace{2mm}
\caption{(a) a twist; (b) a left corner; (c) a right corner.}
\label{fig:topodefs}
\label{fig:leftright}
 \end{minipage} 
\end{figure}

\subsection{Unicellular maps, tour, and canonical rotation system}

{\bf Tour of a unicellular map.} Let $\m$ be a unicellular map. By definition, $\m$ has a unique face. 
The \emph{tour} of the map $\m$ is done by  following the edges of $\m$ starting from the root corner along the distinguished side of the root half-edge, until returning to the root-corner. Since $\m$ is unicellular, every corner is visited once during the tour. 
An edge is said  \emph{two-ways} if it is followed in two different directions during the tour of the map (this is always the case on orientable surfaces), and is said \emph{one-way} otherwise. 
The tour induces an \emph{order of appearance} on the set of corners, for which the root corner is the least element. We denote by $c< d$ if the corner $c$ appears before the corner $d$ along the tour. Lastly, given an orientation convention, a corner is said \emph{left} if it lies on the left of the walker during the tour of map, and \emph{right} otherwise (Figure~\ref{fig:leftright}).\\


\noindent {\bf Canonical rotation-system.}
As explained above, the rotation system associated to a map is defined up to the choice of an orientation convention. We now explain how to choose a particular convention which will be well-suited for our purposes. A map is said \emph{precubic} if all its vertices have degree $1$ or $3$, and its root-vertex has degree 1. 
Let $\m$ be a precubic unicellular map. Since the vertices of $\m$ all have an odd degree, there exists a unique orientation convention at each vertex such that the number of left corners is more than the number of right corners (indeed, flipping a vertex change its left corners into right corners and \emph{vice versa}). We call \emph{canonical} this orientation convention.  From now on, \emph{we will always use the canonical orientation convention}. This defines canonically a rotation system, a set of twists, and a set of left/right corners. Observe that the root corner is a left corner (as is any corner incident to a leaf) and that vertices of degree~$3$ are incident to either  $2$ or $3$ left corners. We have the following additional property. 

\begin{lemma}\label{lemma:edgeways}
In a (canonically oriented) precubic unicellular map, two-ways edges are incident to left corners only and are not twists. 
\end{lemma}

\begin{proof}
Let $e$ be a two-ways edge, and let $c_1, c_2$ be two corners incident to the same vertex and separated by~$e$ ($c_1$ and $c_2$ coincide if that vertex has degree 1).  Since $e$ is two-ways, the corners $c_1$, $c_2$ are either simultaneously left or simultaneously right. By definition of the canonical orientation, they have to be simultaneously left. Thus two-way edges are only incident to left corners. Therefore two-ways edges are not twists since following a twisted edge always leads from a left corner to a right corner or the converse. 
\end{proof}

\subsection{Intertwined nodes.} \label{subsection:intertwined}
We now define a notion of \emph{intertwined node} which generalizes the definition given in \cite{Chapuy:PTRF} for precubic maps on orientable surfaces. 
\begin{definition}\label{def:intertwined}
Let $\m$ be a (canonically oriented) precubic unicellular map, let $v$ be a vertex of degree 3, and let $c_1$, $c_2$, $c_3$ be the incident corners in counterclockwise order around $v$, with the convention that $c_1$ is the first of these corners to appear during the tour of $\m$. The vertex $v$ is called an \emph{intertwined node} if $c_3$ appears before $c_2$ during the tour of $\m$. 

Moreover, we say that the vertex $v$ has flavor {\A} if it is incident to three left corners. Otherwise, $v$ is incident to exactly one right corner, and we say that $v$ is of flavor \B, \C, or {\D} respectively, according to whether the right corner is $c_1$, $c_2$ or $c_3$.
\end{definition}

Observe that the definition of the canonical orientation was a prerequisite to define intertwined nodes. 
The intertwined of some unicellular maps on the Klein bottle are indicated in Figure~\ref{fig:klein}. 
We will now show that intertwined nodes are exactly the ones whose deletion decreases the type of the map without disconnecting it nor increasing its number of faces. 

The \emph{opening} of an intertwined node of a map $\m$ is the operation consisting in splitting this vertex into three (marked) vertices of degree $1$, as in Figure~\ref{fig:opening}. 
That is, we define a rotation system and set of twists of the embedded graph $\n$ obtained in this way (we refrain from calling it a map yet, since it is unclear that it is connected) as the rotation system and set of twists inherited from the original map $\m$. 
\begin{figure}
\begin{minipage}{.48\linewidth}
\vspace{4.5mm}
\centerline{\includegraphics[scale=.8]{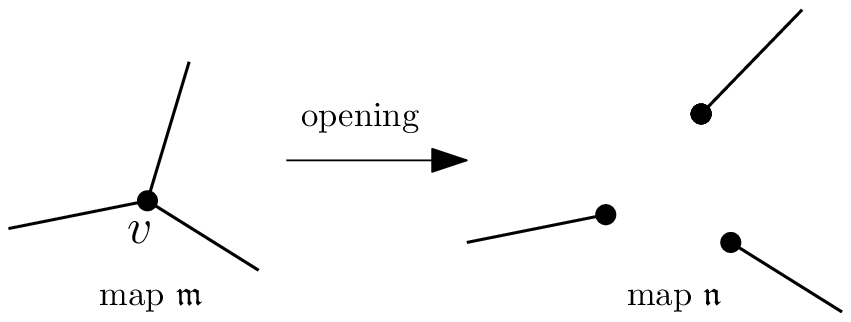}}
\vspace{3mm}
\caption{Opening of an intertwined node (of a precubic map).}
\label{fig:opening}
\end{minipage}\hfill 
\begin{minipage}{.40\linewidth}
\vspace{5mm}
\centerline{\includegraphics[scale=1]{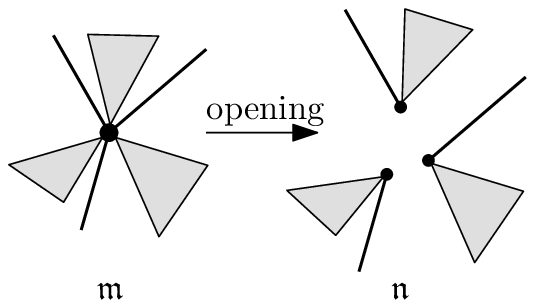}}
\caption{Opening for non-precubic unicellular maps (dominant case).}
\label{fig:openingdominant}
\label{fig:twist}
 \end{minipage} 
\end{figure} 
\begin{samepage}
\begin{proposition}\label{prop:opening}
Let $n$ be a positive integer and let $h$ be in $\{1,3/2,2,5/2,\ldots\}$.
For each flavor $\F$ in $\{\A,\B,\C,\D\}$, the opening operation gives a bijection between the set of precubic unicellular maps with $n$ edges, type $h$, and a distinguished intertwined node of flavor $\F$, and the set of precubic unicellular maps with $n$ edges, type $h-1$ and three distinguished vertices of degree $1$.
The converse bijection is called the \emph{gluing of flavor $\F$}. 

Moreover, if a precubic unicellular map $\m$ is obtained from a precubic unicellular map~$\n$ (of lower type) by a gluing of flavor $\F$, then $\m$ is orientable if and only if $\n$ is orientable and $\F=\A$.
\end{proposition}
The opening of intertwined nodes of type {\A} and {\B} are represented in Figure~\ref{fig:checkopeningAB}.
\end{samepage}

\begin{figure}
\centerline{\includegraphics[scale=.8]{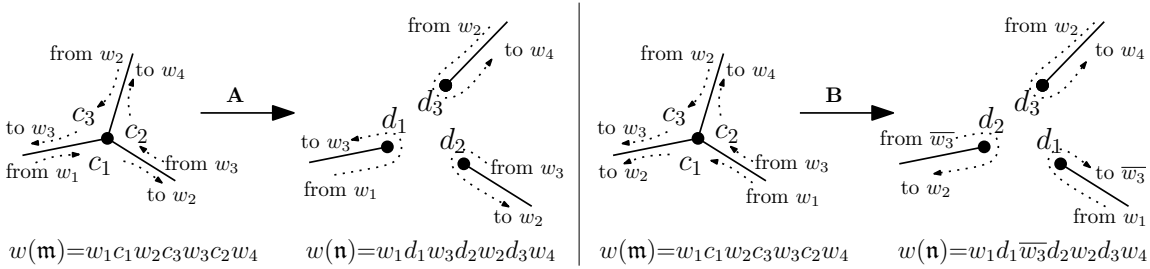}}
\caption{The tours of $\m$ and $\n$, in the case of flavor {\A}, and in the case of flavor \B.}
\label{fig:checkopeningAB}
\end{figure}

\begin{proof}
We first show that the opening of an intertwined vertex produces a unicellular map (and decreases the type by 1). Let $\m$ be a precubic unicellular map, and let $v$ be an intertwined node. Let $c_1, c_2,c_3$ be the three corners incident to $v$ in counterclockwise order, with the convention that $c_1$ is the first of these corners to appear during the tour of $\m$. Since $v$ is intertwined, the sequence of corners appearing during the tour of $\m$ has the form  
$$w(\m)=w_1 c_1 w_2 c_3 w_3 c_2 w_4,$$
where $w_1,w_2,w_3,w_4$ are sequences of corners. Let $\n$ be the embedded graph with marked vertices $v_1,v_2,v_3$ obtained by opening $\m$.  We identify the corners of $\m$ distinct from $c_1,c_2,c_3$ with the corners of $\n$ distinct from the corners $d_1,d_2,d_3$ incident to $v_1,v_2,v_3$. By following the edges of $\n$ starting from the root corner along the distinguished side of the root half-edge, one gets a sequence of corners $w(\n)$. If $v$ has flavor \A, this sequence of corners is
$$w(\n)=w_1 d_1 w_3 d_2 w_2 d_3 w_4,$$
as can be seen from Figure~\ref{fig:checkopeningAB}.
Similarly, if $v$ has flavor \B~ (resp. \C, \D) then the sequence of corner is 
$$w(\n)=w_1 d_1 \overline{w}_3 d_2 w_2 d_3 w_4,~~(\textrm{resp. } w(\n)=w_1 d_1 w_3 d_2 \overline{w}_2 d_3 w_4,~~ w(\n)=w_1 d_1 \overline{w}_2 d_2 \overline{w}_3 d_3 w_4),$$
where $\overline{w}_i$ is the \emph{mirror} of the sequence $w_i$  obtained by reading $w_i$ backward. 
In each case, the sequence $w(\n)$  contains all the corners of $\n$, implying that $\n$ is a unicellular map.  
Moreover, $\n$ has two more vertices than $\m$, so by Euler formula, its type is $h(\n)=h(\m)-1$.\\ 



We now define the \emph{gluing} operation (of flavor \A, \B, \C~or \D) which we shall prove to be the inverse of the opening operation (on node of flavor \A, \B, \C~or \D). 
Let us treat in details the gluing of flavor \B; the other flavors being similar.  
Let $\n$ be a precubic unicellular map with three distinguished leaves $v_1,v_2,v_3$ encountered in this order during the tour of~$\n$. For $i=1,2,3$ we denote by $e_i$ and $c_i$ respectively the edge and corner incident to $v_i$.
We consider the canonical orientation convention of $\n$. Clearly, $e_1,e_2,e_3$ are two-way edges, hence they are not twists for this convention (by Lemma~\ref{lemma:edgeways}).  The \emph{gluing of flavor \B} on the map $\n$ gives a map $\m$ defined as follows: the graph of $\m$ is the graph of $n$ after identification of the three leaves $v_1,v_2,v_3$ into a single vertex $v$, the rotation system of $\m$ is the same as the rotation system of $\n$ at any vertex distinct from $v$, and the rotation system at $v$ is $(e_1,e_3,e_2)$ in counterclockwise order, lastly the set of twists of $\m$ is the set of twists of $\n$ together with the two edges $e_1$ and $e_2$. 
We now prove that the map $\m$ is unicellular. Let us denote by 
$$w(\n)=w_1 d_1 w_2 d_2 w_3 d_3 w_4$$
the sequence of corners encountered during the tour of $\n$ (where the $w_i$ are sequences of corners distinct from $d_1,d_2,d_3$). Let us denote by $c_1,c_2,c_3$ the corners of the new vertex $v$ incident to $(e_1,e_3)$, $(e_3,e_2)$, $(e_2,e_1)$ respectively . 
By following the edges of $\m$ starting from the root corner along the distinguished side of the root half-edge, one gets  the sequence of corners
$$w(\m)=w_1 c_1 w_3 c_3 \overline{w}_2 c_2 w_4,$$
as can be seen from Figure~\ref{fig:checkopeningAB}. This sequence contains all corners of $\m$ showing that $\m$ is unicellular. We will now show that $v$ is an intertwined vertex of flavor \B. 
Observe first that the corners in the sequence $w_1$ are followed in the same direction during the tour of  $\m$  and $\n$, so that the corners in $w_1$ are left corners in the map $\n$ (for its canonical orientation convention) if and only if they are left corners in the map $\m$ (for its non-canonical orientation convention inherited from $\n$).
In particular, the corner preceding $c_1$ during the tour of $\m$ (the last corner in the sequence $w_1$) is a left corner since it is a left corner in $\n$ (indeed, $e_1$ is a two-way edge in $\n$ incident only to left corners by Lemma~\ref{lemma:edgeways}). Since $e_1$ is a twist of $\m$, this implies that $c_1$ is a right corner of $\m$ (for its non-canonical convention). A similar reasoning shows that $c_2$ and $c_3$ are left corners of $\m$ (for its non-canonical convention).
Since $v$ is incident to a majority of left corners, the orientation convention at $v$ is the canonical one. Hence $c_1,c_2,c_3$ are in counterclockwise order around $v$ for the canonical orientation convention of $\m$, which together with the expression of $w(\m)$ shows that $v$ is an intertwined node of flavor~\B.\\

It only remains to prove that the opening of a node of flavor \B~and the gluing of flavor \B~are reverse operations. The reader might already be convinced of this fact by reasoning in terms of ribbon graphs. Otherwise, the proof (which must deal with some orientation conventions) runs as follows. 

We first prove that opening a glued map gives the original map. Let $\n$ be a map with marked leaves $v_1,v_2,v_3$, let $\m$ be the map with new vertex $v$ obtained by the corresponding gluing of flavor \B, and let $\n'$ be the map obtained by opening $\m$ at $v$. It is clear that the graph $G$ underlying $\n$ and $\n'$ is the same and we want to prove that $\n=\n'$ (that is, there exists a set of vertices $U$ such that flipping $U$ changes the system of rotation and set of twists of $\n$ to those of $\n'$). The map $\m$ inherits an orientation convention from $\n$ which might differ from its canonical convention. These two conventions on $\m$ differ by the flipping of a certain subset of vertices $U$, and gives two different systems of rotations and two sets of twists for $\m$. By definition, the map $\n$ and $\n'$ have graph $G$ and rotation system and set of twists given by the non-canonical and canonical convention for $\m$. Observe now that one can get the rotation system and set of twists of $\n$ to those of $\n'$ by flipping the set of vertices $U'$, where $U'=U$ if  $v\notin U$ and $U'=U\setminus \{v\}\cup \{v_1,v_2,v_3\}$ otherwise. Hence, $\n=\n'$. 

We now consider a map $\m$ with intertwined vertex $v$ of flavor \B, the map $\n$ with marked leaves obtained from the opening of $\m$ at $v$, and the map $\m'$ obtained by the gluing of flavor \B. It is clear that the graph $G$ underlying $\m$ and $\m'$ is the same and we want to prove that $\m=\m'$ (that is, there exists a set of vertices $U$ such that flipping $U$ changes the system of rotation and set of twists of $\m$ to those of $\m'$).  The map $\n$ inherits an orientation convention from $\m$ which might differ from its canonical convention. Let  $v_1,v_2,v_3$ be the marked leaves of $\n$ appearing in this order  during the tour of $\n$ and let $e_1,e_2,e_3$ be the incident edges. In the orientation convention $C$ of $\n$ inherited from $\m$, the corners incident to $v_1$ and $v_2$ are right corners, while $v_3$ is a left corner (see Figure~\ref{fig:checkopeningAB}). In the orientation convention $C'$ of $\n$ obtained from the canonical convention by flipping $v_1$ and $v_2$, the corners incident to $v_1$ and $v_2$ are right corners, while $v_3$ is a left corner. This implies that one goes from the convention $C$ to the convention $C'$ by flipping a subset of vertices $U$ not containing $v_1,v_2,v_3$. By definition, the maps $\m$ and $\m'$ have graph $G$ and system of rotation and twists inherited respectively from the conventions $C$ and $C'$ on $\m$. Since the orientation convention of $\m$ and $\m'$ coincide at $v$ (the edges $e_1,e_3,e_2$ appear in this counterclockwise order around $v$) and $v_1,v_2,v_3\notin U$, one gets from the system of rotation and twists of $\m$ to those of $\m'$ by flipping the set of vertices $U$. Hence, $\m=\m'$.
\end{proof}

\section{Main results.}
\label{sec:main}

\subsection{The number of precubic unicellular maps.}

In this section, we present our main results, which rely on two ingredients. The first one is Proposition~\ref{prop:opening}, which enables us to express the number of precubic unicellular maps of type $h$ with a marked intertwined node in terms of the number of unicellular maps of a smaller type. The second ingredient is the fact (to be discussed in Section~\ref{sec:averaging}) that, among maps of type $h$ and fixed size, the average number of intertwined nodes in a map is $2h-1$.\\ 

In order to use Proposition~\ref{prop:opening}, we first need to determine the number of way of choosing non-root leaves in precubic maps.
\begin{lemma}\label{lemma:numberleaves}
Let $h\in\frac{1}{2}\NN$ and let $\m$ be a precubic unicellular map of type $h$. 
Then, the number of edges of $\m$ is at least $6h-1$ and is odd if the type $h$ is an integer and even otherwise. Moreover, if $\m$ has  $2m+\mathbbm{1}_{h\in\NN}$ edges, then it has $m+1-3h-\frac{1}{2}\Ind_{h\notin \NN}$ non-root leaves.
\end{lemma}

\begin{proof}
Let $n_1$ and $n_3$ be the number of vertices of degree $1$ and $3$ in $\m$, respectively. One has $n_1+n_3=v(\m)$ and $n_1+3n_3=2e(\m)$. Moreover, Euler formula gives $v(\m)=e(m)+1-2h$. Solving this system of equations gives $n_1=e(\m)/2+3/2-3h$. Since $n_1\geq 1$  (because the root vertex of a precubic maps is a leaf) this implies the stated conditions on $e(\m)$. Moreover the number of non root leaves  is $n_1-1=m+1-3h-\frac{1}{2}\Ind_{h\notin \NN}$.
\end{proof}


Let $h\geq 1$ be an element of $\frac{1}{2}\NN$, and let $m\geq 1$ be an integer. We denote by $\mathcal{O}_h(m)$ and $\mathcal{N}_h(m)$ respectively the sets of orientable and non-orientable precubic unicellular maps of type $h$ with $2m+\mathbbm{1}_{h\in\NN}$ edges, and we denote by  $\xi_h(m)$ and $\eta_h(m)$ their cardinalities. From Lemma~\ref{lemma:numberleaves} and Proposition~\ref{prop:opening}, the number $\eta_h^{\inter}(m)$ of \emph{non-orientable} unicellular precubic maps of type $h$ with $n$ edges and a marked intertwined node is given by:
\begin{eqnarray}\label{eq:trisectioncount}
\eta_h^\inter(m)  = 4 {\ell \choose 3} \eta_{h-1}(m) +  3 { \ell\choose 3} \xi_{h-1}(m),
\end{eqnarray}
where $\ell=m+4-3h-\frac{1}{2}\Ind_{h\notin \NN}$ is the number of non-root leaves in precubic unicellular maps of type $h-1$ having  $2m+\mathbbm{1}_{h\in\NN}$ edges. Here, the first term accounts for intertwined nodes obtained by gluing three leaves in a non-orientable map of type $h-1$ (in which case the flavor of the gluing can be either {\A}, {\B}, {\C} or {\D}), and the second term corresponds to the case where the starting map of type $h-1$ is orientable (in which case the gluing has to be of flavor {\B}, {\C} or {\D} in order to destroy the orientability).\\ 

The keystone of this paper, to be proved in Section~\ref{sec:averaging} is the following result:
\begin{proposition}\label{prop:average}
There exists and involution $\Phi$ of $\mathcal{N}_h(m)$ such that for all maps $\m\in\mathcal{N}_h(m)$, the total number of intertwined nodes in the maps $\m$ and $\Phi(\m)$ is $4h-2$.
 In particular, the average number of intertwined nodes of elements of $\mathcal{N}_h(m)$ is $(2h-1)$, and one has $\eta_h^{\inter}(m) = (2h-1) \eta_h(m)$.
\end{proposition}
It is interesting to compare Proposition~\ref{prop:average} with the analogous result in~\cite{Chapuy:PTRF}: in the orientable case, \emph{each} map of genus $h$ has exactly $2h$ intertwined nodes, whereas here the quantity $(2h-1)$ is only an \emph{average value}. For example, Figure~\ref{fig:klein} shows two maps on the Klein bottle ($h=1$) which are related by the involution $\Phi$: they have respectively $2$ and $0$ intertwined nodes.\\

As a direct corollary of Proposition~\ref{prop:average} and Equation~\Eref{eq:trisectioncount}, we can state our main result:
\begin{theorem}\label{thm:recursion}
The numbers $\eta_h(m)$ of non-orientable  precubic unicellular maps of type $h$ with $2m+\mathbbm{1}_{h\in\NN}$ edges obey the following recursion:
\begin{align}\label{eq:main}
(2h-1) \cdot \eta_h(m) = 4 {\ell\choose 3} \eta_{h-1}(m) +  3 {\ell\choose 3} \xi_{h-1}(m),
\end{align}
where $\ell=m+4-3h-\frac{1}{2}\Ind_{h\notin \NN}$, and
where $\xi_h(m)$ is the number of orientable precubic unicellular maps of genus $h$ with $2m+\mathbbm{1}_{h\in\NN}$ edges, which is $0$ if $h\not\in\NN$, and is given by the following formula otherwise~\cite{Chapuy:unicellular}:
\begin{eqnarray}\label{eq:orientable}
\xi_h(m) = \frac{1}{(2h)!!} {m+1 \choose 3, 3, \dots, 3,  m+1-3h} \mathrm{Cat}(m)=\frac{(2m)!}{12^hh! m!(m+1-3h)! }. 
\end{eqnarray}
\end{theorem}

Explicit formulas for the numbers $\eta_h(m)$ can now be obtained by iterating the recursion given in Theorem~\ref{thm:recursion}.
\begin{corollary}[the case $h\in\NN$]\label{cor:integer}
Let $h\in\NN$ and $m\in\NN$, $m\geq 3h-1$. Then the number of non-orientable precubic unicellular maps of type $h$ with $2m+1$ edges equals:
\begin{align}\label{eq:integer}
\eta_h(m) =  c_h {m+1 \choose 3, 3, \dots, 3, m+1-3h} \mathrm{Cat}(m) = \frac{c_h \cdot (2m)!}{6^h m!(m+1-3h)! }
\end{align}
where $c_h=\displaystyle 3 \cdot 2^{3h-2} \frac{h!}{(2h)!} \sum_{l=0}^{h-1} {\ 2 l\ \choose l}16^{-l}$.
\end{corollary}
\begin{corollary}[the case $h\not\in\NN$]\label{cor:noninteger}
Let $h\in\{\frac{1}{2},\frac{3}{2},\frac{5}{2},\ldots\}$ and $m\in\NN$, $m\geq 3\floor{h}-1$. Then the number of non-orientable precubic unicellular maps of type $h$ with $2m$ edges equals:
\begin{eqnarray*}
\eta_h(m) &=& \frac{4^{\floor{h}}}{(2h-1)(2h-3)\dots 1} {m-1 \choose 3, 3, \dots, 3, m -1 -3\floor{h}} \times \eta_{1/2}(m)  \\
          &=& \frac{4^{m+\floor{h}-1} (m-1)!}{6^{\floor{h}} (2h-1)!! (m-1-3\floor{h})!}.
\end{eqnarray*}
\end{corollary}
\begin{proof}[Proof of Corollary~\ref{cor:integer}]
It follows by induction on $h$ and Equations~\Eref{eq:main} and~\Eref{eq:orientable} that the statement of Equation~\Eref{eq:integer} holds, with the constant $c_h$ defined by the recurrence 
$c_0=0$ (since there is no non-orientable map of type 0) and $c_h = \la_{h-1} c_{h-1}+a_{h-1}$, with $\la_{h-1}=\frac{4}{2h-1}$ and  $a_{h-1}=\frac{3}{(2h-1)(2h-2)!!}=\frac{3}{2^{h-1}(h-1)!(2h-1)}$. 
The solution of this recurrence is
$
c_h = \sum_{l=0}^{h-1} a_{l} \la_{l+1} \la_{l+2}\dots \la_{h-1}.
$
 Now, by definition, the product
$a_{l} \la_{l+1} \la_{l+2}\dots \la_{h-1}$ equals $\frac{3\cdot 4^{h-1-l} }{2^ll!(2l+1) (2l+3)(2l+5)\dots(2h-1)}.
$
Using the expression $\frac{1}{(2l+1) (2l+3)\dots(2h-1)} =\frac{2^hh!}{2h!} \cdot \frac{(2l)!}{2^l l!}$ and reporting it in the sum gives the expression of~$c_h$ given in Corollary~\ref{cor:integer}.
\end{proof}

\begin{proof}[Proof of Corollary~\ref{cor:noninteger}]
Since for non-integer $h$ we have $\xi_{h-1}(m)=0$, the first equality is a direct consequence of an iteration of Theorem~\ref{thm:recursion}. Therefore the only thing to prove is that the number $\eta_\half(m)$ of precubic maps in the projective plane is $4^{m-1}$. This can be done by induction on $m$ via an adaptation of R\'emy's bijection~\cite{remy}. For $m=1$, we have $\eta_\half(m)=1$ (the only map with two edges is made of an edge joining the root-leaf to a node and of a twisted loop incident to this node). For the induction step, observe that precubic projective unicellular  maps with one distinguished non-root leaf are in bijection with precubic  projective unicellular maps with one leaf less and a distinguished edge-side: too see that, delete the distinguished leaf, transform the remaining vertex of degree $2$ into an edge, and remember the side of that edge on which the original leaf was attached. Since a projective precubic unicellular map with $2m$ edges has $m-1$ non-root leaves and $4m$ edge-sides, we obtain for all $m\geq1$ that $m\, \eta_\half(m+1)=4m\, \eta_\half(m)$, and the result follows.
\end{proof}

\noindent \textbf{Remark.} Before closing this subsection we point out that the bijection \emph{\`a la R\'emy} presented in the proof of Corollary~\ref{cor:noninteger} can be adapted to any surface. Such a bijection implies that the number $\eta_{h}(m)$ of maps on surface of integer type $h$ satisfies $(m+1-3h)\eta_h(m)=2(2m-1)\eta_h(m-1)$. This recursion on $m$ imposes the general form of the numbers $\eta_{h}(m)$:
$$\eta_{h}(m)=\frac{2m(2m-1)}{m(m-3h+1)}\eta_{h}(m-1)=\,\ldots\,=K_h\frac{(2m)!}{m!(m-3h+1)!}$$
for the constant $K_h=\frac{m_{\min}!}{(2m_{\min})!}\nu_h(m_{\min})$, where $m_{\min}=3h-1$. Similarly, for non-integer type $h$, one gets $(m-1-3\floor{h})\eta_{h}(m)=2(2m-2)\eta_{h}(m-1)$, hence
$$\eta_{h}(m)=\frac{4(m-1)}{(m-1-3\floor{h})}\eta_{h}(m-1)=\,\ldots\,=K_h\frac{4^m(m-1)!}{(m-1-3\floor{h})!},$$
for the constant $K_h=\frac{1}{4^{3\floor{h}+1}m_{\min}!}\eta_h(m_{\min})$, where $m_{\min}=3\floor{h}+1$. Observe however that the approach \emph{\`a la R\'emy} is not sufficient to determine explicitly the value of $\eta_h(m_{\min})$.\\ 

\subsection{The asymptotic number of general  unicellular maps.}
In this subsection we derive the asymptotic number of \emph{arbitrary} (i.e., non-necessarily precubic) unicellular maps of given type (the type is fixed, the number of edges goes to infinity). The expression of these asymptotic numbers was already given in \cite{BernardiRue} in terms of the  number of unicellular \emph{cubic maps} (maps with vertices of degree 3) of the same type (see also \cite{ChMaSc} for the orientable case). Moreover, the number of unicellular cubic maps of type $h$ is easily seen to be $\nu_h(m_{\min})$, where $m_{\min}=3h-1+\frac{1}{2}\Ind_{h\notin \NN}$, hence can be obtained from Corollaries~\ref{cor:integer} and~\ref{cor:noninteger}. The goal of this subsection is rather to explain how to adapt the \emph{opening bijections} described in Subsection~\ref{subsection:intertwined} to (almost all) arbitrary  unicellular maps.\\

Let $\m$ be a unicellular map of type $h$. The  \emph{core} of $\m$ is the map obtained by deleting recursively all the leaves of $\m$ (until every vertex has degree at least 2). Clearly, the core is a unicellular map of type $h$ formed by paths of vertices of degree $2$ joining vertices of degree at least $3$. The \emph{scheme} of $\m$ is the map obtained by replacing each of these paths by an edge, so that vertices in the scheme have degree at least $3$. We say that a unicellular map is \emph{dominant} if the scheme is \emph{cubic} (every vertex has degree 3).
\begin{proposition}[\cite{ChMaSc,BernardiRue}]\label{prop:dominant}
Let $h\in\frac{1}{2}\NN$. Then, among non-orientable unicellular maps of type $h$ with $n$ edges, the proportion of maps which are dominant tends to $1$ when $n$ tends to infinity.
\end{proposition}
The idea behind that proposition is the following. Given a scheme $\s$, one can easily compute the generating series of all unicellular maps of scheme $\s$ (there is only a finite number of schemes), by observing that these maps are obtained by substituting each edge of the scheme with a path of trees. A generating function approach then easily shows that the schemes with maximum number of edges are the only one contributing to the asymptotic number of unicellular maps. These schemes are precisely the cubic ones.\\

The opening bijection of Subsection~\ref{subsection:intertwined} can be adapted to dominant unicellular maps as follows. Given a dominant map $\m$ of type $h$ and scheme $\s$, and $v$ an intertwined node of $\s$, we can define the opening operation of $\m$ at $v$ by splitting the vertex $v$ in three, and deciding on a convention on the redistribution of the three ``subtrees" attached to the scheme at this point (Figure~\ref{fig:openingdominant}): one obtains a dominant map $ \n$ of type $h-1$ with three distinguished vertices. These vertices are not \emph{any} three vertices: they have to be in \emph{general position} in $\n$ (i.e.,  they cannot be part of the core, and none can lie on a path from one to another), but again, in the asymptotic case this does not make a big difference: when $n$ tends to infinity, the proportion of triples of vertices which are in general position tends to $1$. 
We do not state here the asymptotic estimates that can make the previous claims precise (they can be copied almost verbatim from the orientable case~\cite{Chapuy:PTRF}), but rather we state now our asymptotic theorem:
\begin{theorem}\label{thm:asymptotic}
Let $\kappa_h(n)$ be the number of non-orientable rooted unicellular maps of type $h$ with $n$ edges. Then one has, when $n$ tends to infinity:
$$
(2h-1) \kappa_h(n) \sim  4 \frac{n^3}{3!} \kappa_{h-1}(n) + 3\frac{n^3}{3!} \epsilon_{h-1}(n)
$$
where $\epsilon_h(n)$ denotes the number of orientable rooted unicellular maps of genus $h$ with $n$ edges. Therefore,
$$
\kappa_h(n)\sim_{n\to \infty} \frac{c_h}{\sqrt{\pi}6^h} n^{3h-\frac{3}{2}} 4^n \ \mbox{ if }h\in\NN, \ \ \ 
\kappa_h(n)\sim_{n\to \infty} \frac{4^{\floor{h}}}{2 \cdot 6^{\floor{h}}(2h-1)!!} n^{3h-\frac{3}{2}} 4^n \ \mbox{ if }h\not\in\NN.
$$ 
where the constant $c_h$ is defined in Corollary~\ref{cor:integer}.
\end{theorem}

\medskip


\section{The average number of intertwined nodes}\label{sec:averaging}

In this section we prove Proposition~\ref{prop:average} stating that the average number of intertwined nodes among precubic unicellular maps of type $h$ and size $m$ is exactly $(2h-1)$:
\begin{eqnarray}\label{eq:etainter}
\eta_h^{\inter}(m) = (2h-1) \eta_h(m).
\end{eqnarray}

Let us emphasize the fact that the number of intertwined nodes is not a constant over the set  of unicellular precubic maps of given type and number of edges. For instance, among the six maps with 5 edges on the Klein bottle $\mN_1$, three maps have 2 intertwined nodes, and three maps have none; see  Figure~\ref{fig:mapsKlein}. 
As stated in Proposition~\ref{prop:average}, our strategy to prove Equation~\Eref{eq:etainter} is to exhibit an involution $\Phi$ from the set $\mathcal{N}_{h}(m)$ to itself, such that  for any given map $\m$, the total number of intertwined nodes in the maps $\m$ and $\Phi(\m)$ is $4h(\m)-2$. Observe from Figure~\ref{fig:mapsKlein} that the involution $\Phi$ cannot be a simple re-rooting of the map $\m$. 

\begin{figure}[h]
\centerline{\includegraphics[scale=.7]{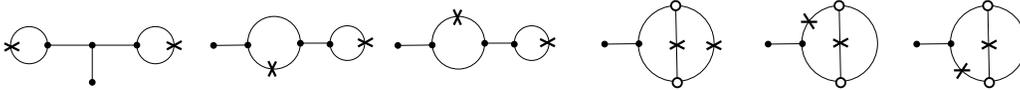}}
\caption{The precubic unicellular maps with 5 edges on the Klein bottle (the root is incident to the unique leaf). Intertwined nodes are indicated as white vertices.}
\label{fig:mapsKlein}
\end{figure} 

Before defining the mapping $\Phi$, we relate the number of intertwined nodes of a map to certain properties of its twists.  Let $\m$ be a (canonically oriented) precubic map, and let $e$ be an an edge of $\m$ which is a twist. Let $c$ be the corner incident to $e$ which appears first in the tour of $\m$. We say that $e$ is \emph{left-to-right} if $c$ is a left-corner, and that it is \emph{right-to-left} otherwise (see Figure~\ref{fig:klein}). In other words, the twist $e$ is left-to-right if it changes the side of the corners from left, to right, when it is crossed for the first time in the tour of the map (and the converse is true for right-to-left twists). 
\begin{lemma}
\label{lemma:trisection}
Let $\m$ be a precubic unicellular map of type $h(\m)$, considered with its canonical orientation convention. Then, its numbers  $\tau(\m)$ of intertwined nodes, $\TLR(\m)$ of left-to-right twists, and $\TRL(\m)$ of right-to-left twists are related by:
\begin{eqnarray}\label{eq:trisectionlemma}
 \tau(\m)= 2 h(\m)+ \TRL(\m) - \TLR(\m).
\end{eqnarray}
\end{lemma}

\begin{figure}
\centerline{\includegraphics[scale=.8]{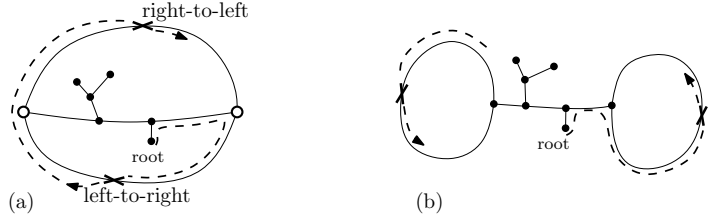}}
\caption{Two maps on the Klein Bottle $\mN_1$ and their intertwined nodes (white vertices). The number of twists are (a) $\TLR(\m)=1$, $\TRL(\m)=1$; (b) $\TLR(\m)=2$, $\TRL(\m)=0$.}
\label{fig:klein}
\end{figure} 

\begin{proof}
We first define the \emph{label} of a corner of $\m$ as the element of $\{1,\ldots,2e(\m)\}$ indicating the position of appearance of this corner during the tour of the map: the root corner has label $1$, the corner that follows it in the tour has label $2$, etc...
We say that a corner of $\m$ is a \emph{descent} if it is followed, counterclockwise around its vertex, by a corner of smaller or equal canonical label, and that it is an \emph{ascent} otherwise. We let $\dsc$ and $\asc$ be the total numbers of descents and ascents in $\m$, respectively. We will now compute the difference $\dsc - \asc$ in two different ways: one by summing over edges, the other by summing over vertices (extending the ideas used in~\cite{Chapuy:unicellular} for the orientable case).

To each edge $e=(v_1,v_2)$ of $\m$ we associate the two (distinct) corners $c_1,c_2$ incident to the vertices $v_1$ and $v_2$ respectively and following $e$ clockwise around their vertex. Clearly this, creates a partition of the set of corners of $\m$ and we can compute $\dsc - \asc$ by adding the contribution of each edge.
Let $e=(v_1,v_2)$ be an edge. For $i=1,2$ we denote by $l_i$ the label of $c_i$ and by $l_i'$ the label of the other corner incident to $v_i$ and $e$ ($l_i'=l_i$ if $v_i$ is a leaf).  
Up to exchanging $v_1$ and $v_2$, we can assume that $l_1<l_2$.
 We now examine five cases:\\
$\bullet$ $e$ is the root edge of $\m$; Figure~\ref{fig:trisectionproof}(a). Since $v_1$ is a leaf, the edge $e$ is two-ways hence not a twist and incident to left-corners only. We get $l_1=l_1'=1$,  $l_2=2e(\m)$, and $l_2'=2$. Hence  both $c_1$ and $c_2$ are descents.\\
$\bullet$ $e$ is not the root edge, is not a twist, and is two-ways; Figure~\ref{fig:trisectionproof}(b). In this case, we know by Lemma~\ref{lemma:edgeways} that all the corners are \emph{left}, from which $(l_1',l_2')=(l_2+1,l_1+1)$. Therefore $c_1$ is an ascent and $c_2$ is a descent.\\
$\bullet$ $e$ is not the root edge, is not a twist, and is one-way; Figure~\ref{fig:trisectionproof}(c). In this case,
we have  $(l_1',l_2')=(l_2+\epsilon,l_1-\epsilon)$, where $\epsilon=1$ or $-1$ according to whether $c_1$ is a left or a right corner. In both cases, $c_1$ is an ascent and $c_2$ is a descent (observe that $v_1$ cannot be a leaf by Lemma 2, hence $l_1\neq l_1'$).\\
$\bullet$ $e$ is a right-to-left twist; Figure~\ref{fig:trisectionproof}(d). In this case,  $l_1<l_2$ implies that $e$ is followed from $v_1$ to $v_2$.  
By definition of a right-to-left twist, $l_1'<l_1$ hence $c_1$ is an descent. Moreover, $l_2=l_1+1$ and $l_2'=l_1'+1$, therefore $c_2$ is also a descent.\\
$\bullet$ $e$ is a left-to-right twist; Figure~\ref{fig:trisectionproof}(e). This case is similar to the previous one. The corner $c_1$ is necessarily ascent, and since $l_2=l_1+1$ and $l_2'=l_1'+1$, the corner $c_2$ is also an ascent.
\begin{figure}[h]
\centerline{\includegraphics[scale=.9]{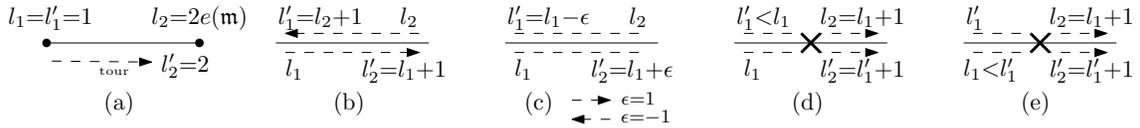}}
\caption{The five cases of the proof of Lemma~\ref{lemma:trisection}.}
\label{fig:trisectionproof}
\end{figure}

Expressing the difference $\dsc-\asc$ as a sum over all edges of $\m$, we obtain from the five cases above:
\begin{align*} 
\dsc - \asc = 2 + 0+ 0+ 2\TRL(\m) - 2\TLR(\m)  = 2 \Big(1 + \TRL(\m) - \TLR(\m)\Big).
\end{align*}
Using the fact that $\dsc+\asc=2e(\m)$, we obtain the total number of descents in $\m$ which is  $\dsc= e(\m) + 1 + \TRL(\m) - \TLR(\m)$.

Now, there is another way of counting the descents. Indeed, since by definition each non-intertwined node has exactly one descent, and each intertwined node has exactly two of them, one gets:
$\displaystyle \dsc= 2\tau(\m) + (v(\m)-\tau(\m))$.
Solving for $\tau(\m)$ and using the previous expression for $\dsc$ gives
\begin{align*}
\tau(\m)=\dsc - v(\m)= e(\m)+1 -v(\m) + \TRL(\m) - \TLR(\m).
\end{align*}
The lemma then follows by applying Euler's formula.
\end{proof}


We now define the promised mapping $\Phi$ averaging the number of intertwined nodes.  Let $\m$ be a unicellular precubic map on a non-orientable surface. We consider the canonical orientation convention for the map $\m$, which defines a rotation system and set of twists. The set of twists is non-empty since the map $\m$ lives on a non-orientable surface.
By cutting every twist of $\m$ at their middle point, one obtains  a graph together with a rotation system and some \emph{dangling half-edges} that we call \emph{buds}.
The resulting embedded graph with buds, which we denote by $\hm$, can have several connected components and each component (which is a map with buds) can have several faces; see Figure~\ref{fig:involution-qui-tue}. We set a convention for the direction in which one \emph{turns around a face} of $\hm$: the edges are followed in such a way that every corner is left (this is possible since $\hm$ has no twist). For any bud $b$ of $\hm$, we let $\sigma(b)$ be the bud following~$b$ when turning around the face of~$\hm$ containing~$b$. Clearly, the mapping $\sigma$ is a permutation on the set of buds. We now define $\Phi(\m)$ to be the graph with rotation system and twists obtained from $\hm$ by gluing together into a twist the buds $\sig(b)$ and $\sig(b')$ for every pair of buds $b,b'$ forming a twist of $\m$. The mapping $\Phi$ is represented in Figure~\ref{fig:involution-qui-tue}.\\
\begin{figure}[h!]
\centerline{\includegraphics[scale=.7]{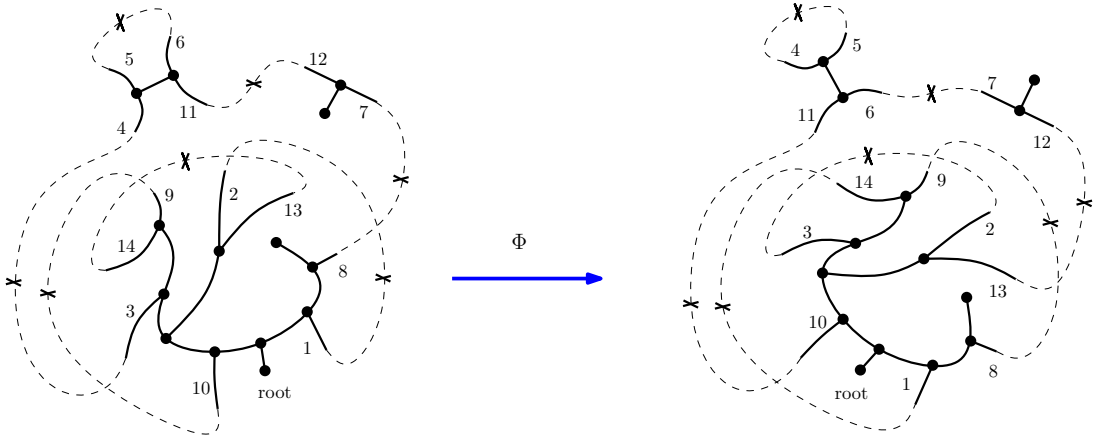}}
\caption{A unicellular map $\m$ and its image by the mapping $\Phi$. The twists are indicated by (partially) dotted lines, while the map $\hm$ is represented in solid lines.}
\label{fig:involution-qui-tue}
\end{figure}

Before proving that $\Phi(\m)$ is a unicellular map, we set some additional notations.
We denote by $k$ the number of twists of $\m$ and we denote by  
$w(\m)=w_1w_2\cdots w_{2k+1}$
the sequence of corners encountered during the tour of $\m$, where the subsequences $w_i$ and $w_{i+1}$ are separated by the traversal of a twist for $i=1\ldots 2k$. Observe that corners in $w_i$ are left corners of $\m$ if $i$ is odd, and right corners if $i$ is even (since following a twist leads from a left to a right corner or the converse). Hence, the sequence of corners encountered between two buds around a face of $\hm$ are one of the sequences $w_1',w_2',\ldots, w_{2k}'$, where $w_1'=w_{2k+1}w_1$, and for $i>1$,  $w_i'=w_i$ if $i$ is odd and $w_i'=\overline{w}_i$ otherwise (where $\overline{w}_i$ is the \emph{mirror} sequence of $w_i$ obtained by reading $w_i$ backwards). We identify the buds of $\hm$ (i.e. the half-twists of $\m$ or $\Phi(m)$) with the integers in $\{1,\ldots,2k\}$ by calling $i$ the bud following the sequence of corners $w_i'$ around the faces of $\hm$. This labelling is indicated in  Figure~\ref{fig:involution-qui-tue}. We will now consider the permutation $\sigma$ as a permutation on $\{1,\ldots,2k\}$ and we denote $r=\sigma^{-1}(1)$. The map in Figure~\ref{fig:involution-qui-tue} gives $\sig=(1,8,13,2,9,14,3,10)(4,11,6,5)(7,12)$ and $r=10$. We first prove a technical lemma.

\begin{lemma}\label{lem:parity}
The permutation $\sigma$ maps odd to even integers. In particular, $r=\sigma^{-1}(1)$ is even.
\end{lemma}
\begin{proof}
By Lemma~\ref{lemma:edgeways}, all twists of $\m$ are one-way. Hence, every bud of $\hm$ is incident both to a left corner and to a right corner of $\m$. The lemma therefore follows from the fact that left and right corners of $\m$ belong to the sequences $w_i'$ for $i$ odd and $i$ even respectively. 
\end{proof}
 
We are now ready to prove that $\Phi(\m)$ is unicellular and a little more. In the following, we denote by $\modu{i}$ the representative of an integer $i$ modulo $2k$ belonging to $\{1,\ldots, 2k\}$.

\begin{lemma}\label{lem:tour-Psi(m)} 
The embedded graph $\Phi(\m)$ is a unicellular map. Moreover, the rotation system and set of twists of $\Phi(\m)$ inherited from $\m$ correspond to the canonical orientation convention of $\Phi(\m)$. 
Lastly,  the sequence of corners encountered during the tour of $\Phi(\m)$ reads 
$v_1v_2\ldots v_{2k+1},$
where the subsequences $v_i$ separated by twist traversals are given by $v_i=w_{\sig(\modu{r+1-i})}$ for all $i=1,\ldots, 2k$, and $~v_{2k+1}=w_{2k+1}$.
\end{lemma}



\begin{proof}
We consider, as above, the map $\m$ with its canonical orientation convention and the map $\Phi(\m)$ with the orientation convention inherited from $\m$. We denote by $\alpha$ the (fixed-point free) involution on $\{1,\ldots,2k\}$ corresponding to the twists of $\m$. That is to say, $\alpha(i)=j$ if the half-edges $i,j$ form a twist of $\m$. We also denote by  $\be=\si\al\si^{-1}$ the involution corresponding to the twists of $\Phi(\m)$. 

\noindent \textbf{Fact 1:}  For $i\in\{1,\ldots, 2k\}$, $\al(i)=i+1$ if $i$ is odd (hence, $\al(i)=i-1$ if $i$ is even). Similarly $\beta(i)=\modu{i+1}$ if $i$ is even  (hence $\be(i)=\modu{i-1}$ if $i$ is odd).\\ 
To prove Fact 1, recall that $w_1,\ldots,w_n$ denote the sequences of corners, encountered in that order during the tour of $\m$. If $i\in\{2,3,\ldots,2k\}$ is odd (resp. even), then the sequence of left corners $w_i=w_i'$ (resp. right corners $w_i=\overline{w}_i'$) goes from the bud~$\sig^{-1}(i)$ to the bud~$i$ (resp. from the bud $i$ to the bud $\sig^{-1}(i)$) during the tour of $\m$; see Figure~\ref{fig:fact1}. Hence, if $i\in\{1,2\ldots 2k\}$ is odd, the twist of $\m$ traversed between $w_i$ and $w_{i+1}$ is made of the half-twists $i$ and $i+1$, while if $i$ is even it is made of the half-twists  $\sig^{-1}(i)$ and $\sig^{-1}(\modu{i+1})$. 
From the odd case, one gets $\al(i)=i+1$ if $i$ is odd. From the even case, one gets $\al(\si^{-1}(i))=\si^{-1}(\modu{i+1})$ if $i$ is even. That is, $\beta(i)\equiv\si\al\si^{-1}(i)=\modu{i+1}$ if $i$ is even. 
\begin{figure}[h]
\begin{minipage}{.4\linewidth}
\centerline{\includegraphics[scale=.8]{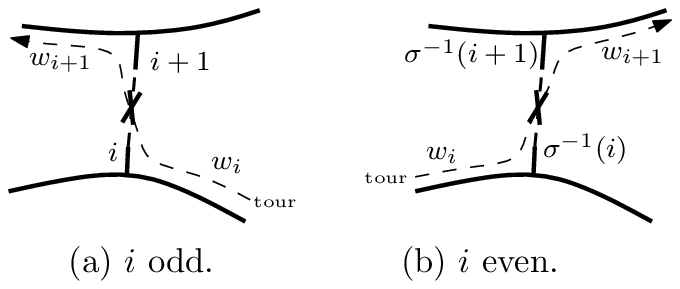}}
\caption{Proof of Fact 1.}
\label{fig:fact1}
\end{minipage} \hfill
\begin{minipage}{.4\linewidth}
\centerline{\includegraphics[scale=.8]{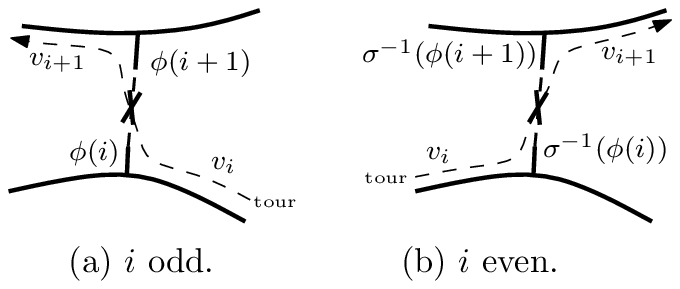}}
\caption{Proof of Fact 2.}
\label{fig:fact2}
\end{minipage}
\end{figure} 

We now denote by  $v_1v_2v_3\ldots v_{\ell+1}$ the sequence of corners encountered by following the edges of $\Phi(\m)$ starting and ending at the root corner (tour of the face of $\Phi(\m)$ containing the root),  where the subsequences $v_i$ and $v_{i+1}$ are separated by a twist traversal. Clearly, $v_1=w_1$,  $v_{\ell+1}=w_{2k+1}$ and for $i=1\ldots \ell+1$ the corners in $v_i$ are left corners of $\Phi(\m)$ if and only if $i$ is odd. For $i=1\ldots \ell$, we denote $v_i'=v_i$ if $i$ is odd and $v_i'=\overline{v}_i$ otherwise, so that each of the sequences $v_1'$ belongs to $\{w_1',\ldots,w_{2k}'\}$. For $i=1\ldots \ell$, we denote by $\phi(i)$ the bud following $v_i'$ around the faces of $\hm$. Then, the same reasoning as above (see Figure~\ref{fig:fact2}) proves:\\
\noindent \textbf{Fact 2:} For $i=1,\ldots, \ell-1$,  $\beta(\phi(i))=\phi(i+1)$ if $i$ is odd, and $\sig\beta\sig^{-1}(\phi(i))=\phi(i+1)$ if $i$ is even.\\

\noindent \textbf{Fact 3:} If  $\phi(i)=\sig(j)$ for certain integers $1\leq i< \ell$ and $1<j\leq 2k$ of different parity, then  $\phi(i+1)=\sig(\modu{j-1})$.\\
The Fact 3 is easily proved by the following case analysis. If $i$ is odd, then 
$$\phi(i+1)=\beta(\phi(i))=\beta(\sig(j))=\sig\al(j)=\si(j-1),$$
where the first and last equalities are given by Fact~2 and Fact~1 respectively. Similarly,  if $i$ is even 
$$\phi(i+1)=\sig\beta\sig^{-1}(\phi(i))=\sig\beta\sig^{-1}(\sig(j))=\sig\beta(j)=\sig(\modu{j-1}).$$

We now consider the relation  $\phi(1)=1=\sig(r)$ and recall that 1 and $r$ are of different parity by Lemma~\ref{lem:parity}. Then Fact 3 implies by induction that $\phi(i)=\si(\modu{r+1-i})$ for $i=1\ldots \ell$.  This proves that $v_i'=w'_{\si(\modu{r+1-i})}$ for $i=1\ldots\ell$. Since $i$ and $\si(\modu{r+1-i})$ have the same parity (by Lemma~\ref{lem:parity}), this also gives $v_i=w_{\si(\modu{r+1-i})}$ for $i=1\ldots \ell$. In particular, for $i=\ell$, one gets  $v_\ell=w_{\si(\modu{r+1-\ell})}$. 
Moreover, by definition $v_{\ell+1}=w_{2k+1}$, hence $v_{\ell}=w_{\si\beta(r)}$. Hence, $\beta(r)=\modu{r+1-\ell}$. Since $r$ is even, $\beta(r)=r+1$ and $\ell=2k$. 

The sequence $v_1v_2\ldots v_{2k+1}$ contains all the corners of  $\Phi(\m)$. Hence, $\Phi(\m)$ is a unicellular map. Moreover, a corner is left for the map $\m$ (resp. the map $\Phi(\m)$ considered with its orientation convention inherited from $\m$) if and only if it belongs to a sequence $w_i$ (resp. $v_i=w_{\si(\modu{r+1-i})}$) for an odd integer $i$. Since $i$ and $\si(\modu{r+1-i})$ have the same parity a corner is left for $\m$  if and only if it is left for $\Phi(\m)$.  This shows that the orientation convention of $\Phi(\m)$ inherited from $\m$ is the canonical convention of $\Phi(\m)$.
\end{proof}




We now make the final strike by considering the action of $\Phi$ on the set $\mathcal{N}_h(m)$ of non-orientable maps of type $h$.

\begin{proposition}\label{prop:Psi-average}
Let $m$ be a positive integer and $h$ be in $\{1/2,1,3/2,\ldots\}$. 
The mapping $\Phi$ is a bijection from the set $\mathcal{N}_h(m)$ to itself. Moreover, for every map $\m$ in $\mathcal{N}_h(m)$, the total number of intertwined nodes in the maps $\m$ and $\Phi(\m)$ is $4h-2$.
\end{proposition}

\begin{proof}
Clearly, the maps $\m$ and $\Phi(\m)$ have the same number of edges and vertices. Hence, they have the same type by Euler formula. Moreover, they both have $k>0$ twists (for their canonical convention) hence are non-orientable. Thus, $\Phi$ maps the set $\mathcal{N}_h(m)$ to itself. To prove the bijectivity (i.e. injectivity) of $\Phi$, observe that for any map $\m$, the embedded graphs $\hm$ and $\widehat{\Phi(\m)}$ are equal; this is because the canonical rotation system and set of twists of $\m$ and $\Phi(\m)$ coincide. In particular, the permutation $\sig$ on the half-twists of $\m$ can be read from $\Phi(\m)$. Hence, the twists of $\m$ are easily recovered from those of $\Phi(\m)$: the buds $i$ and $j$ form a twist of $\m$ if   $\sig(i)$ and $\sig(j)$ form a twist of $\Phi(\m)$.

We now proceed to prove that the total number of intertwined nodes in $\m$ and $\Phi(\m)$ is $4h-2$. By Lemma~\ref{lemma:trisection}, this amounts to proving that  $\TLR(\m) - \TRL(\m)+\TLR(\Phi(\m))  - \TRL(\Phi(\m))=2$. Since $\m$ and $\Phi(\m)$ both have $k$ twists, $\TLR(\m)- \TRL(\m)+\TLR(\Phi(\m))  - \TRL(\Phi(\m))=2(\TLR(\m)+\TLR(\Phi(\m))-k)$. Hence, we have to prove $\TLR(\m)+\TLR(\Phi(\m))=k+1$.

Let $i$ be a bud of $\hm$, let $t$ be the twist of $\m$ containing $i$, and let $c,c'$ be the corners preceding and following $i$ in counterclockwise order around the vertex incident to $i$. By definition, the twist $t$ of $\m$  is left-to-right if and only if $c$ appears before $c'$ during the tour of $\m$. Given that the corners $c$ and $c'$ belong respectively to the subsequences $w_i$ and $w_{\sig(i)}$ (except if $i=r$ in which case $\si(i)=1$ and $c'$ is in $w_{2k+1}$), the twist $t$ is left-to right if and only if $i<\sig(i)$ or $i=r$. 

Let us now examine under which circumstances the bud $\sig(i)$ is part of a left-to-right twist of $\Phi(\m)$. 
The corners  $d$ and $d'$ preceding and following the bud $\si(i)$ in counterclockwise order around the vertex incident to $\si(i)$ belong respectively to  $w_{\si(i)}$ and $w_{\si\si(i)}$ (except if $\si(i)=r$, in which case $\si\si(i)=1$ and $c'$ belongs to $w_{2k+1}$). By Lemma~\ref{lem:tour-Psi(m)},  $w_{\si(i)}=v_{\modu{r+1-i}}$ for $i=1\ldots 2k$.   Therefore, the twist $t'$ of $\Phi(\m)$ containing $\si(i)$ is left-to-right  (for $\hm$) if and only if $\modu{r+1-i}<\modu{r+1-\si(i)}$ or $\si(i)=r$. 

The two preceding points gives the number $\TLR(\m)+\TLR(\Phi(\m))$ of left-to right twists as 
$$\TLR(\m)+\TLR(\Phi(\m))=1+\textstyle\frac{1}{2}\sum_{i=1}^{2k}\delta(i),$$
where $\delta(i)=\mathbbm{1}_{i<\si(i)}+\mathbbm{1}_{\modu{r+1-i}<\modu{r+1-\si(i)}}$ is the sum of two indicator functions (the factor $1/2$ accounts for the fact that a twist has two halves). The contribution $\delta(i)$ is equal to 2 if  $i\leq r< \sig(i)$,~ 0 if  $\sig(i)\leq r < i$, and~1 otherwise.
 Finally, there are as many integers $i$ such that $i\leq r< \sig(i)$ as integers such that $\sig(i)\leq r < i$ (true for each cycle of $\sig$). Thus, $\sum_{i=1}^{2k}\delta(i)=2k$, and $\TLR(\m)+\TLR(\Phi(\m))=k+1$.
\end{proof}

The last Proposition is sufficient to establish Equation~\Eref{eq:etainter}, and the enumerative results of Section~\ref{sec:main}. However, Proposition~\ref{prop:average} was saying a little bit more, namely that the bijection $\Phi$ can be chosen as an \emph{involution}:

\begin{proof}[Proof of Proposition~\ref{prop:average}]
Observe that, as we defined it, the bijection $\Phi$ is not an involution. But one can easily define an involution from $\Phi$, as the mapping acting as $\Phi$ on elements $\m$ of $\mathcal{N}_h(m)$ such that $\tau(\m)>2h-1$, acting as $\Phi^{-1}$ if $\tau(\m)<2h-1$, and as the identity if $\tau(\m)=2h-1$.
\end{proof}



\bibliographystyle{alpha}
\bibliography{biblio-unicellular}
\label{sec:biblio}
\end{document}